\documentclass{amsart}
\usepackage{amssymb, latexsym}

\newcommand{\beqa}{\begin{eqnarray*}}
\newcommand{\eeqa}{\end{eqnarray*}}
\newcommand{\beqn}{\begin{eqnarray}}
\newcommand{\eeqn}{\end{eqnarray}}

\newcommand{\D}{\mathbb D}
\newcommand{\R}{\mathbb R}

\newcommand{\N}{\mathbb N}

\newcounter{cnt1}
\newcounter{cnt2}
\newcounter{cnt3}
\newcommand{\blr}{\begin{list}{$($\roman{cnt1}$)$}
 {\usecounter{cnt1} \setlength{\topsep}{0pt}
 \setlength{\itemsep}{0pt}}}
\newcommand{\bla}{\begin{list}{$($\alph{cnt2}$)$}
 {\usecounter{cnt2} \setlength{\topsep}{0pt}
 \setlength{\itemsep}{0pt}}}
\newcommand{\bln}{\begin{list}{$($\arabic{cnt3}$)$}
 {\usecounter{cnt3} \setlength{\topsep}{0pt}
 \setlength{\itemsep}{0pt}}}
\newcommand{\el}{\end{list}}
\newtheorem{thm}{Theorem}[section]
\newtheorem{lem}[thm]{Lemma}
\newtheorem{cor}[thm]{Corollary}

\newtheorem{Def}[thm]{Definition}
\newtheorem{prop}[thm]{Proposition}
\newtheorem{rem}[thm]{Remark}
\newcommand{\Rem}{\begin{rem} \rm}
\newcommand{\bdfn}{\begin{Def} \rm}
\newcommand{\edfn}{\end{Def}}

\newcommand{\ba}{\begin{array}}
\newcommand{\ea}{\end{array}}

\begin{document}

\begin{center}\large{{\bf{  Non-absolute integrable function spaces on metric measure spaces    }}} 
		\vspace{0.05cm}
		
	  Parthapratim Saha$^{a},$  Bipan Hazarika$^{b},$ Hemanta Kalita$^{c}$	  
		
				\vspace{0.05cm}
				\vspace{0.5cm}
				$^a$ Department of Mathematics, Sipajhar College, Darrang, Assam, India\\
				Email:   parthasaha.sipclg@gmail.com, parthapratimsaha@sipajharcollege.ac.in \\
$^b$ Department of Mathematics, Gauhati University, Guwahati, Assam, India\\
Email:  bh\_gu@gauhati.ac.in \\
$^c$ Mathematics Division, VIT Bhopal University, Kothri-kalan, Sehore, Bhopal-Indore Highway, India\\
				Email: hemanta30kalita@gmail.com
	
\end{center}
	\title{}
	\author{}
	\begin{abstract} 
Kuelbs-Steadman spaces are introduced in this article on a separable metric space having finite diameter together with a finite Borel measure. Kuelbs-Steadman spaces of the Lipschitz type are also discussed. Various inclusion properties are also discussed.  In the sequel, we introduce HK-Sobolev spaces on metric mesure space which coincides with HK-Sobolev space in the Euclidean case.  In application, we discuss the boundedness of Hardy-Littlewood maximal operator on Kuelbs-Steadman spaces and HK-Sobolev spaces over a metric measure space.    \\
		
		 \\
		\textbf {2020 Mathematical Subject Clasification}:  46B25, 46E35, 46E36, 46F25.
	\end{abstract}
	\maketitle
	
	
	\pagestyle{myheadings}
	\markboth{\rightline {\scriptsize  P. Saha, B. Hazarika, H. Kalita}}
	{\leftline{\scriptsize  Non-absolute integrable function spaces on metric measure spaces  }}
	
	\maketitle 


 \section{Introduction}
 Many recent research have focused on Kuelbs-Steadman spaces (see, for example, the references in \cite{GILLSURVEY, Gill}, and \cite{HK, KHM}).
The concept behind studying these spaces is to think of the $L^1$ spaces as containing the Henstock-Kurzweil integrable functions in a sense within a bigger Hilbert space with a smaller norm.
This makes it possible to apply a wide range of mathematical concepts to functional analysis and other areas of study, including  quantum physics, Fourier transforms, convolution operators, Feynman integrals,  differential equations, Markov chains \cite{GILLSURVEY, Gill, KHM}, Gaussian measures (see also \cite{KUELBS}).

This method also enables the development of a functional analysis theory that connects Kuelbs-Steadman spaces with Sobolev-type spaces rather than with traditional $L^p$ spaces. In $20^{th}$ century, finding the solution to the Dirichlet and Neuman problems for the Laplace equation was one of the most significant mathematical physics problems (\cite{WM}).
Famous scientists at that time, including Hilbert, Courant, Weyl, and many others, were fascinated by this subject. In 1930, the major challenge of this problem was resolved by  S. Sobolev, who proposed a functional space known as the Sobolev space, which is defined by maps in $\textit{L}^p(\mathbb{R}^n)$ whose distributional derivatives of order up to $\mathtt{k}$ exist and  contained in $\textit{L}^p(\mathbb{R}^n).$ (readers can see \cite{Sobolev}).  Among the  essential tools of functional analysis, Sobolev spaces are one of them. They are used in an assortment of approaches to solve ordinary or partial differential equations or difference equations.  \cite{Brezis, TW}. B. Hazarika et al. \cite{BH} introduced a Sobolev type spaces containing non-absolute integrable functions, associate with Kuelbs-Steadman spaces so called HK-Sobolev spaces. Sobolev spaces are subspaces of HK-Sobolev spaces (see \cite{BH}). The Poincar\'e inequality, motivate us to develope Poincar\'e type inequality on HK-Sobolev spaces. To execute our motivation, we extent Kuelbs-Steadman spaces on metric measure space too.

The article is structured as follows: in Section 2, basic concepts and terminology are introduced along with some definitions and findings.
Kuelbs-Steadman type spaces on metric measure spaces were first discussed in Section 3.
We explore Lipschitz type Kuelbs-Steadman spaces in Subsection 3.1. In additional, several inclusion properties are discussed in this section.
In fourth section, HK-Sobolev spaces and Lipschitz type HK-Sobolev spaces are explored on $\mathbb{R}^n$ and metric measure space $X,$ respectively.
In last section, we discuss the boundedness of maximal Hardy operators on Kuelbs-Steadman spaces and HK-Sobolev spaces, respectively.  

\section{Preliminaries}
 Throughout the article, with $\mu$ being a locally finite (i.e. finite on bounded sets) Borel regular measure on $\mathcal{X},$ we call  $(\mathcal{X},d)$ a separable metric space.
 
 There must be a constant $L \geq 0$ which satisfy the following condition for all $s, t \in \mathcal{X}$ $$ |\mathfrak{f}(s)-\mathfrak{f}(t)| \leq L d(s,t) ,$$ when  $\mathfrak{f}: \mathcal{X} \to \mathbb{R}$ is a Lipschitz function,

$\it Lip(\mathcal{X})$ stands for the set of Lipschitz functions on ${\mathcal{X}}$ and  lowest such constant $L$ called Lipschits constant, is denoted by $Lip(\mathfrak{f}).$
Whenever there are multiple points in $\mathcal{X}$, the following semi-norm is added to the space $Lip(\mathcal{X})$. $$l(\mathfrak{f})=  \sup\limits_{x, y \in \mathcal{X}} \frac{|\mathfrak{f}(t)-\mathfrak{f}(s)|}{d(t,s)}, ~where~ s \neq t,~\mathfrak{f} \in Lip(\mathcal{X}).$$ This doesn't constitute a norm because $\mathfrak{f}=0$ only if $\mathfrak{f}$ is constant (see \cite{Stud}). We refer to the quotient space $\frac{Lip(\mathcal{X})}{ const(\mathcal{X})}$ as $LIP(\mathcal{X})$ where $const(\mathcal{X})$ is the set of all real-valued constant maps on $\mathcal{X}$. As far as its norm $L(\mathfrak{f}+const(\mathcal{X}))=l(\mathfrak{f}),~\mathfrak{f} \in Lip(\mathcal{X})$ is concerned, the space $LIP(\mathcal{X})$ is a Banach space.  Also, recalling $LIP(\mathcal{X})$ will be a function space if for $x,~y$ are two points of $\mathcal{X},$ such that the functional $\mathfrak{f} \to \mathfrak{f}(x)-\mathfrak{f}(y) \in \mathbb{R}$ is well defined in $LIP(\mathcal{X}).$ For a Lipschitz function $\mathfrak{f}$ the modulus of gradient is given by the slope $|\nabla \mathfrak{f}|: \mathcal{X} \to \R,$ by 
\begin{eqnarray*}
|\nabla \mathfrak{f}|(x) = \lim\limits_{y \to x} \sup \frac{|\mathfrak{f}(y)-\mathfrak{f}(x)|}{d(y,x)}.
\end{eqnarray*}

  Recalling Sobolev spaces in the approach of Lipschitz as 
  \begin{eqnarray*}
  W^{1,p}(\Omega) = \bigg\{\mathfrak{f} \in D^\prime(\Omega):~\mathfrak{f} \in L^p(\Omega),~\nabla \mathfrak{f} \in L^p(\Omega)\bigg\}, \\
  L^{1,p}(\Omega)= \bigg\{\mathfrak{f} \in D^\prime(\Omega):~\nabla \mathfrak{f} \in L^p(\Omega)\bigg\}
\end{eqnarray*} 
where $\Omega \subseteq \R^n$ is an open set and $ 1 \leq p \leq \infty.$ It is well known that  ${W^{1,p}}(\R_I^n)$ is a Banach space with the norm $\|\mathfrak{f}\|_{{W^{1,p}}}= \|\mathfrak{f}\|_{{L^p}}+ \|\nabla \mathfrak{f}||_{{L^p}}$ and $L^{1,p}(\Omega)$ is endowed with a semi norm $\|\mathfrak{f}\|_{L^{1,p}} = \|\nabla \mathfrak{f}\|_{L^p}.$

 The equality of $L^{1,p}(\Omega)$ and $ W^{1,p}(\Omega)$ are found from the followings:
\begin{lem}
\cite[Lemma 7.16]{Gilbert} If $\mathfrak{f} \in L^{1,p}(Q)$, with $Q$ a cube in $\mathbb{R}^n$, then $$|\mathfrak{f}(x)-\mathfrak{f}_Q| \leq C\int_Q\frac{|\nabla \mathfrak{f}(y)|}{|x-y|^{n-1}}dy~a.e.~and ~L^{1,p}(\Omega)= W^{1,p}(\Omega),$$
where $\mathfrak{f}_q=\mu(q)^{-1}\int_Q \mathfrak{f} d\mu$ is the average value of $\mathfrak{f}$ over $Q$ and $C$ is a constant.
\end{lem}
\begin{Def}
The domain $ \Omega \subseteq \R^n$ is said to have the extension property whether there is a bounded linear operator $E: W^{1,p}(\Omega) \to W^{1,p}(\R^n),$ such that for every $u \in W^{1,p}(\Omega),~Eu {| \Omega}= u $ a.e.
\end{Def}
\begin{prop}\cite{Gilbert}
If $\Omega$ is a bounded domain with the extension property then $L^{1,p}(\Omega)= W^{1,p}(\Omega).$
\end{prop}
The Hardy-Littlewood Maximal Operator for a locally integrable function $f$ on  $\mathcal{X}$ is provided by $\mathcal{M} \mathfrak{f} : \mathcal{X} \rightarrow [0, \infty]$ so that
\begin{equation}
\mathcal{M}\mathfrak{f}(x)= \sup\limits_{x \in B}\frac{1}{\mu(B)}
\int_{B}|\mathfrak{f}(t)|dt
\end{equation}
where $B$ is a ball in $\mathcal{X}$ containing the point $\mathcal{X}$ and  supremum is taken over all such ball. It is well known that the maximal operator is bounded in $L^p(\mathbb{R}^n)$ for $1<p\leq \infty$ (see \cite{Stein}). In \cite{Mf} for a doubling metric space $\mathcal{X}$, J. Heinonen proved that the maximal operator is bounded  in $L^p(\mathcal{X})$ for $1< p \leq \infty$ . Juha Kinnunem in the paper \cite{Kinnunem} proved that the maximal operator is also bounded in  Sobolev space $W^{1,p}(\mathbb{R}^n)$ for $1< p\leq \infty$. 
\begin{Def}
\cite{Mf} If there exist a positive constant $C$ such that $\mu(2B) \leq C \mu(B)$  for all balls $B$ on a space $\mathcal{X}$, then the measure $\mu$ on $\mathcal{X}$ is called doubling. $C$ is called doubling constant.
\end{Def}
\begin{thm}[Basic Covering Theorem]\cite{Mf} \label{BCL}
Let $\mathcal{F}$ be a family of balls with uniformly bounded diamete in a metric space $\mathcal{X}$. Then $\exists$ a disjointed subfamily $\mathcal{G}$ with the following properties: every ball $B \in \mathcal{F}, ~\exists $  a ball $B'$ in $\mathcal{G}$ such that 
\begin{equation}
B \cap B' \neq \phi~~~ \mbox{and}~~ ~radius(B') \geq \frac{1}{2} radius(B)
\end{equation}
In fact, \begin{equation}
\bigcup\limits_{B \in \mathcal{F}}B \subset \bigcup\limits_{B \in \mathcal{G}}5B
\end{equation}

\section{ Kuelbs-Steadman spaces on metric measure spaces}
 We introduce, in this section the theory of  Kuelbs-Steadman spaces on metric measure space. We also introduce  Lipschitz-type Kuelbs-Steadman spaces which in brif, we will call  as Lipschitz-Steadman spaces. We'll talk about these spaces' fundamental characteristics.. Finally, we will find the inclusion relations of Lipschitz-Steadman spaces and Lipschitz-type Lebesgue spaces $L^{1,p}({\R}^n).$ Here we recalling the Kuelbs-Steadman spaces $\mathcal{KS}^p(\R^n)$

 We discuss a few results of $\mathcal{KS}^p(\R^n)$ that are useful in a later section of our work. For more details see \cite{Gill}.
 \begin{thm}
 For $ 1\leq p \leq \infty,~D^\prime(\R^n)$ is a subset of $\mathcal{KS}^p(\R^n).$
  \end{thm}
  
  The followimg theorem is a particular case  of   \cite[theorem 3.0.4]{HK}
  \begin{thm}
    For $1 < p \leq \infty$, we have $\|M_l \mathfrak{f} \|_{\mathcal{KS}^p(\mathbb{R}^n)} \leq C_p \|\mathfrak{f}\|_{\mathcal{KS}^p(\mathbb{R}^n)}$ for all $\mathfrak{f} \in \mathcal{KS}^p(\R^n),$  where $C_p $ depends on $p$ but not $\mathfrak{f}.$
  \end{thm}
 Now, we will extend the theory of Kuelbs-Steadman spaces on spearable metric spaces. Let $(\mathcal{X}, d)$ be separable metric space with a regular Borel measure $\mu$.  Let $\mathcal{S}$ be a countable dense subset of $\mathcal{X}$. Since $\mathcal{S}$ is countable, we can arrange the elements of $\mathcal{S}$ as $\{x_1 , x_2 , x_3,... \}$. Now for each $r, j \in \mathbb{N}$ let $B_r(x_i)$ be the closed ball centered at $x_i$ with radius $k$. Using natural order that translates  $\mathbb{N} \times \mathbb{N}$ bijectively to $\mathbb{N}$, we can consider $\{ B_r , r \in \mathbb{N} \}$ as countable collection of all closed balls whose center lies in $\mathcal{S}.$

 Assuming the characteristic function $\chi_k$ on $B_k$, we have $\chi_k \in \mathit{L}^p(\mathcal{X}) \cap \mathit{L}^\infty(\mathcal{X})$ for $1\leq p < \infty$. Also, we consider a sequence $\{\tau_r\}$  of positive real numbers which satisfy $\sum\limits_{r=1}^\infty \tau_r =1$.  We define 
$$\|\mathfrak{f}\|_{\mathcal{KS}^p(\mathcal{X})} = \left\{\begin{array}{cc}\left(\sum\limits_{r=1}^\infty \tau_r \left|\int_{\mathcal{X}}\chi_r(t) \mathfrak{f}(t)d\mu \right|^p\right)^{\frac{1}{p}} & \mbox{~when~} 1\leq p<\infty \\     
  \sup\limits_{r \geq 1} \left|\int_{\mathcal{X}} \chi_r(t)\mathfrak{f}(t)d\mu(t)\right| & \mbox{~when~} p~ is ~\infty \end{array}\right. $$
  One can easily check that $\|\cdot\|_{\mathcal{KS}^p(\mathcal{X})}$ defines a norm on $L^p(\mathcal{X})$.  The completion of $L^p(\mathcal{X})$ with respect to this norm is called is defined as Kuelbs-Steadman space on $\mathcal{X}$, denoted by $\mathcal{KS}^p(\mathcal{X}).$ If $p=1,$ then $\|\cdot\|_{\mathcal{KS}^1(\mathcal{X})}$ will be a norm of $\mathcal{KS}^1(\mathcal{X}).$ In case, if the topology is weak one, we call $\mathcal{KS}^1(\mathcal{X})$ as weak Kuelbs-Steadman spaces. We denote this spaces as $\mathcal{KS}^1_w(\mathcal{X}).$  One can see \cite{ABH} for details. We will discuss a few fundamental results of $\mathcal{KS}^p(\mathcal{X})$ as follows:
  \begin{thm}
  For every $q,~1 \leq q \leq \infty,$ the Lebesgue space $L^q(\mathcal{X})$ is dense continuous embedding subset of $ \mathcal{KS}^p(\mathcal{X}).$
  \end{thm}
  \begin{proof}
  The space $\mathcal{KS}^p(\mathcal{X})$ is contained in $L^p(\mathcal{X})$ densely by the construction of $\mathcal{KS}^p(\mathcal{X})$. Now, for $q \neq p$, we need to show that $L^p(\mathcal{X}) \subset \mathcal{KS}^p(\mathcal{X})$. We take into account the following situations:
  \begin{enumerate}
  \item For $p= \infty$ the result is trivial.
  \item  For $p < \infty$, we consider $\mathfrak{f} \in L^q(\mathcal{X})$ and $q < \infty$, then
   \begin{align*}
   \|\mathfrak{f}\|_{\mathcal{KS}^p(\mathcal{X})} &= \left\{ \sum\limits_{r=1}^\infty \tau_r \left| \int_\mathcal{X} \chi_r(t)\mathfrak{f}(t)d\mu(t)\right|^{\frac{qp}{q}}\right\}^{\frac{1}{p}}\\
   & \leq \left\{ \sum\limits_{r=1}^\infty \tau_r \left( \int_\mathcal{X} \chi_r(t)|\mathfrak{f}(t)|^qd\mu(t)\right)^{\frac{p}{q}}\right\}^{\frac{1}{p}}\\
   & \leq \sup\limits_r \left( \int_X \chi_r(t)|\mathfrak{f}(t)|^q d\mu(t) \right)^\frac{1}{q} \leq \|\mathfrak{f}\|_q.
   \end{align*}
   $\therefore ~ L^p(\mathcal{X}) \subset \mathcal{KS}^p(\mathcal{X}).$  
   
   \item For $q= \infty $, we have 
   \begin{align*}
   \|\mathfrak{f}\|_{\mathcal{KS}^p} &=  \left[ \sum\limits_{r=1}^\infty \tau_r \left| \int_\mathcal{X} \chi_r(t)\mathfrak{f}(t)d\mu(t)\right|^p \right ]^{\frac{1}{p}}\\
   & \leq \left\{ \left(\sum\limits_{r=1}^\infty \tau_r(\mu(B_r))^p\right)( \operatorname{ess}\sup|\mathfrak{f}|)^p\right\}^{\frac{1}{p}} \leq M \|\mathfrak{f}\|_\infty .  
   \end{align*}
   Thus $L^\infty (\mathcal{X}) \subset \mathcal{KS}^p(\mathcal{X}).$
  \end{enumerate}

  \end{proof}
\begin{rem} Two exponents $p$ and $q$ are called conjugate exponent is they staisfy $\frac{1}{p}+\frac{1}{q} = 1$.
\end{rem}
  \begin{thm}
  \rm(H\"older type inequality) \rm Consider $1\leq p,q < \infty$ are conjugate  exponent. When $\mathfrak{f} \in \mathcal{KS}^p(\mathcal{X})$ and $\mathfrak{g} \in \mathcal{KS}^q(\mathcal{X})$, then $\mathfrak{f},\mathfrak{g} \in \mathcal{KS}^1(\mathcal{X})$ and $\|\mathfrak{fg}\|_{\mathcal{KS}^1} \leq \|\mathfrak{f}\|_{\mathcal{KS}^p}. \|\mathfrak{g}\|_{\mathcal{KS}^q}$
  \end{thm}
  \begin{proof}
  To prove the inequality, we used generalised form of arithmitic-geometric mean inequality: if $A, b \geq 0$, and $0 \leq \theta \leq 1$, then
  \begin{equation}
  A^\theta B^\theta \leq \theta A + (1- \theta)B.
  \end{equation}
  
  If $\|\mathfrak{f}\|_{\mathcal{KS}^p(\mathcal{X})} = 0 $ or $ \|\mathfrak{g}\|_{\mathcal{KS}^q(\mathcal{X})}=0$, then $\mathfrak{fg} = 0~ a.e.$ and  the inequality is obvious. So, we consider neither $\|\mathfrak{f}\|_{\mathcal{KS}^p(\mathcal{X})} =0$ nor $\|g\|_{\mathcal{KS}^q(\mathcal{X})} =0.$ Now if we replace $\mathfrak{f}$ by $\mathfrak{f}/{\|\mathfrak{f}\|_{\mathcal{KS}^p}} $ and $\mathfrak{g}$ by $\mathfrak{g}/ {\|\mathfrak{g}\|_{\mathcal{KS}^q}}$ and assume $\|\mathfrak{f}\|_{\mathcal{KS}^p(\mathcal{X})} = 1$ and $\|g\|_{\mathcal{KS}^q(\mathcal{X})}=1$, we need to show that $\|\mathfrak{f}g\|_{\mathcal{KS}^1} \leq 1.$
  
  Setting $A= |\mathfrak{f}(t)|^p$, $B= |\mathfrak{g}(t)|^q$ and $\theta = 1/p$ so that $1-\theta = 1/q,$ we get
  \begin{equation}\label{Hol}
  |\mathfrak{f}(t)g(t)| \leq \frac{1}{p}|\mathfrak{f}(t)|^p + \frac{1}{q} |\mathfrak{g}(t)|^q .
  \end{equation}
  
  Now, Using inequality (\ref{Hol}) 
  \begin{align*}
  \|\mathfrak{fg}\|_{\mathcal{KS}^1}& = \sum\limits_{r=1}^\infty \tau_r \left| \int_\mathcal{X} \chi_r(t)\mathfrak{f}(t)\mathfrak{g}(t)d\mu(t)\right| \leq \sum\limits_{r=1}^\infty \tau_r  \int_X \chi_r(t)\left|\mathfrak{f}(t)\mathfrak{g}(t)\right| d\mu(t)\\
  &\leq \sum\limits_{r=1}^\infty \tau_r  \int_\mathcal{X} \chi_r(t)\left( \frac{1}{p}|\mathfrak{f}(t)|^p + \frac{1}{q} |\mathfrak{g}(t)|^q \right) d\mu(t).
  \end{align*}
  This conclude that $\|\mathfrak{fg}\|_{\mathcal{KS}^1} \leq 1$ and this completes the proof.
  \end{proof}
  \begin{thm}
  \rm(Minkowski) \rm Sum of two function in  $\mathcal{KS}^p(\mathcal{X})$  belongs to  $ \mathcal{KS}^p(\mathcal{X})$. Moreover, $\|\mathfrak{f}+\mathfrak{g}\|_{\mathcal{KS}^p(\mathcal{X})} \leq ||\mathfrak{f}||_{\mathcal{KS}^p(\mathcal{X})} + ||\mathfrak{g}||_{\mathcal{\mathcal{KS}}^p(\mathcal{X})}.$
  \end{thm}
  \begin{proof} For any $\mathfrak{f}, \mathfrak{g} \in \mathcal{KS}^p(\mathcal{X})$ and $p\neq \infty$, we have
  \begin{align*}
  \|\mathfrak{f}+\mathfrak{g}\|_{\mathcal{KS}^p(\mathcal{X})} &= \left[ \sum\limits_{r=1}^\infty \tau_r \left| \int_\mathcal{X} \chi_r(t)(\mathfrak{f}+\mathfrak{g})(t)d\mu(t) \right|^p\right]^{\frac{1}{p}} \\
  &= \left[ \sum\limits_{r=1}^\infty \tau_r \left|\left\{ \int_\mathcal{X} \chi_r(t)\mathfrak{f}(t)d\mu(t)\right\} + \left\{ \int_\mathcal{X} \chi_r(t)\mathfrak{g}(t)d\mu(t)\right\} \right|^p\right]^{\frac{1}{p}}\\
  &\leq  \left[ \sum\limits_{r=1}^\infty \tau_r \left| \int_\mathcal{X} \chi_r(t)\mathfrak{f}(t)d\mu(t)\right|^p + \sum\limits_{r=1}^\infty \tau_r \left|\int_\mathcal{X} \chi_r(t)\mathfrak{g}(t)d\mu(t)\right|^p\right]^{\frac{1}{p}}\\
  & \leq \|\mathfrak{f}\|_{\mathcal{KS}^p(\mathcal{X})} + \|\mathfrak{g}\|_{\mathcal{KS}^p(\mathcal{X})}.
  \end{align*}
  For $p= \infty$, we have
  \begin{align*}
  \|\mathfrak{f}+\mathfrak{g}\|_{\mathcal{KS}^\infty(\mathcal{X})} &= \sup\limits_{r \geq 1} \left|\int_{\mathcal{X}} \chi_r(t)(\mathfrak{f}+\mathfrak{g})(t)d\mu(t)\right| \\
   &\leq \sup\limits_{r \geq 1} \left|\int_{\mathcal{X}} \chi_r(t)\mathfrak{f}(t)d\mu(t)\right| + \sup\limits_{r \geq 1} \left|\int_{X} \chi_r(t)\mathfrak{g}(t)d\mu(t)\right| = \|\mathfrak{f}\|_{\mathcal{KS}^\infty(X)}+ \|\mathfrak{g}\|_{\mathcal{KS}^\infty(\mathcal{X})}.
  \end{align*}
  \end{proof}
  
  The space  $\mathcal{KS}^p$  transforms into a metric space with metric $d(\mathfrak{f},\mathfrak{g})=\|\mathfrak{f}- \mathfrak{g}\|_ {\mathcal{KS}^p}$ as a result of the triangle inequality. The fundamental analytical truth is that $\mathcal{KS}^p(\mathcal{X})$ is complete in the notion that each Cauchy sequence just like in norm $\|\cdot\|_{\mathcal{KS}p(\mathcal{X})}$ accumulates to an entry in $\mathcal{KS}^p(\mathcal{X})$.
  \begin{thm}
  The space $\mathcal{KS}^p(\mathcal{X})$ is a Banach space in the norm $\|\cdot\|_{\mathcal{KS}^p(\mathcal{X})}.$
  \end{thm}
  \begin{proof}
 Being a completion of $L^p(\mathcal{X})$ with regards to  $\|\cdot\|_{\mathcal{KS}^p(\mathcal{X})}$, the space $\mathcal{KS}^p(\mathcal{X})$ is complete.
  \end{proof}
  To start additional notes, we examine several potential inclusion relations between different $\mathcal{KS}^p(\mathcal{X})$ spaces. If the core space has a finite measure, the situation is straightforward. 
  \begin{prop}
With a finite positive measure, the space $\mathfrak{X}$ satisfies  $\mathcal{KS}^{p_1}(\mathcal{X}) \subset \mathcal{KS}^{p_0}(\mathcal{X})$ for  $p_0 \leq p_1,$. Futhermore, if $\mathfrak{f} \in \mathcal{KS}^{p_1},$ then there exists $c_0$ and $c_1$ such that 
$$c_0 \|\mathfrak{f}\|_{\mathcal{KS}^{p_0}} \leq c_1 \|\mathfrak{f}\|_{\mathcal{KS}^{p_1}},$$
  where $c_0=\frac{1}{(\sum_{k=1}^\infty \mu(B_k))^\frac{1}{p_0}}$ and $c_1= \frac{1}{(\sum_{k=1}^\infty \mu(B_k))^\frac{1}{p_1}}$ and $B_k$ is the closed ball as described in the definition of the space $\mathcal{KS}^p(\mathcal{X}).$
  \end{prop}
  \begin{proof}
  We may assume $p_0 < p_1.$ If $\mathfrak{f} \in \mathcal{KS}^{p_1}(\mathcal{X})$, taking $\mathtt{F}= |\chi_k \mathfrak{f}|^{p_0}$ (where $\chi_k$ is the characteristic map on the ball $B_k$ as in the deinition of $\mathcal{KS}^p$ norm), $\mathtt{G}=1$, $p = \frac{p_1}{p_0} >1$ and $\frac{1}{p}+\frac{1}{q}=1$, we get applying H\"oder's inequality to $\mathtt{F}$ and $\mathtt{G}$ that
  \begin{align*}
  \|\mathfrak{f}\|_{\mathcal{KS}^p}^{p_0} \leq \left( \sum_{k=1}^\infty \int_X \left|\chi_k \mathfrak{f}~ d\mu\right|^{p_1}\right)^{\frac{p_0}{p_1}} \left(\sum_{k=1}^\infty \mu(B_k)\right)^{1-\frac{p_0}{p_1}}
  \end{align*}
taking $p_0^{th}$  root in the above inequality gives the required result.
  \end{proof}

  \subsection{\bf Lipschitz type Kuelb-Steadman spaces}
  We will discuss Lipschitz type Kuelbs-Steadman spaces on separable metric spaces $\mathcal{X}$ (in short $\mathcal{KS}^p(\mathcal{X})$). To understand Lipschitz type $\mathcal{KS}^p(\mathcal{X}),$  we need to understand Lipschitz type $\mathcal{KS}^p$ spaces on $\mathbb{R}^n.$\\
  { We now define the Lipschitz Type $\mathcal{KS}^{1,p}(\R^n)$ as follows:}
  
  \begin{Def}
  Let us define 
  \begin{eqnarray}
 \mathcal{KS}^{1,p}(\mathbb{R}^n) = \bigg\{\mathfrak{f} \in D^\prime(\R^n):~\nabla \mathfrak{f} \in K{S^p}(\R^n)\bigg\}.
  \end{eqnarray}
   \end{Def}
  $\mathcal{KS}^{1,p}(\R^n)$ is endowed with a semi norm $\|\mathfrak{f}\|_{K{S^{1,p}}}=  \|\nabla \mathfrak{f}\|_{\mathcal{KS}^p}.$
  \begin{rem}
  $\|f\|_{\mathcal{KS}^{1,p}}=  \|\nabla \mathfrak{f}\|_{\mathcal{KS}^p}$ will not be a norm, because it annihilates constant functions.
  \end{rem}
  \begin{thm}
  For each $p,~1 \leq p < \infty,~L^{1,p}(\R^n) \subset \mathcal{KS}^{1,p}(\R^n)$ as a dense continuous embedding.
  \end{thm}
  \begin{proof}
    If $f$ is in $ L^{1,\mathtt{p}}(\R^n)$,  then we obtain 
  \begin{align*}
  \|\mathfrak{f}\|_{\mathcal{KS}^{1,p}(\R^n)} &= \|\nabla \mathfrak{f}\|_{\mathcal{KS}^{p}(\R^n)} \\& = \bigg(\sum\limits_{z=1}^{\infty}t_z \left|\int_{\R^n}\mathcal{E}_z(t) \nabla \mathfrak{f} d\mu(t)\right|^p\bigg)^\frac{1}{p} \\& \leq \sum\limits_{z=1}^{\infty} \bigg( t_z \int_{\R^n}\mathcal{E}_z(t) |\nabla \mathfrak{f}|^p d\mu(x)\bigg)^\frac{1}{p} \\& \leq \sup\limits_z \bigg( \int_{\R^n}\mathcal{E}_z(t) |\nabla \mathfrak{f}|^p d\mu(t)\bigg)^\frac{1}{p} \\&\leq \|\mathfrak{f}\|_{L^{1,p}(\R^n)}.
  \end{align*}
 This brings the proof to a close.
  \end{proof}
  \begin{thm}
Whenever $\mathfrak{f}_n \to \mathfrak{f} $ weakly in $L^{1,p}(\R^n)$ for every $ 1 \leq p < \infty,$ then $\mathfrak{f}_n$  converges  strongly to $\mathfrak{f}$. Namely, any weakly compact subset of $L^{1,p}(\R^n) $ is compact in $\mathcal{KS}^{1,p}(\R^n). $  
  
 \end{thm}
 \begin{proof}
 Let $\mathfrak{f}_n \to f$ weakly on compact subset $K \subseteq L^{1,p}(X).$ So, $\|\mathfrak{f}_n -\mathfrak{f}\|_{L^{1,p}(X)} \to 0$ (weakly) on $K \subseteq L^{1,p}(\mathcal{X}).$ So, $\|\mathfrak{f}_n-\mathfrak{f}\|_{L^{1,p}(\mathcal{X})}=\|\nabla(\mathfrak{f}_n-\mathfrak{f})\|_{L^p(\mathcal{X})}$ (weakly) on $K \subseteq L^p(\mathcal{X}).$ Now, from the \cite[Theorem 3.27]{Gill}, $\|\nabla(\mathfrak{f}_n-\mathfrak{f})\|_{\mathcal{KS}^p(\mathcal{X})} \to 0$ (strongly) on $\mathcal{KS}^p(\mathcal{X}).$ Consequently, $\mathfrak{f}_n \to \mathfrak{f}$ strongly on $\mathcal{KS}^{1,p}(\mathcal{X}).$
 \end{proof}
  \begin{thm}
  For each $q,~1 \leq q < \infty,~ \frac{1}{p}+\frac{1}{q}=1, ~L^{1,q}(\R^n) \subset \mathcal{KS}^{1,p}(\R^n)$ as a dense continuous embedding.
  \end{thm}
  \begin{proof}
 Assume that $q \neq p,~p< \infty.$ If we consider $q < \infty$ and $\mathfrak{f}$ belongs to $L^{1,q}(\R^n),$ we obtain
  \begin{align*}
  \|\mathfrak{f}\|_{\mathcal{KS}^{1,p}(\R^n)} &= \bigg( \sum\limits_{z=1}^{\infty} t_z \left| \int_{\R^n}\mathcal{E}_z(\mathtt{x}) \nabla \mathfrak{f}(\mathtt{x})d\mu(\mathtt{x})\right|^\frac{qp}{p}\bigg)^\frac{1}{p} \\& \leq \bigg( \sum\limits_{z=1}^{\infty}t_z \bigg(\int_{\R^n}\mathcal{E}_z(\mathtt{x}) \big|\nabla \mathfrak{f}(\mathtt{x})\big|^q d\mu(\mathtt{x})\big)^\frac{p}{q}\bigg)^\frac{1}{p} \\ &\leq \sup\limits_z\bigg(\int_{\R^n} \mathcal{E}_z(\mathtt{x}) \big|\nabla f(\mathtt{x})\big|^q d\mu(\mathtt{x})\bigg)^\frac{1}{q} \\& \leq ||\nabla f||_{L^{1,q}}.
  \end{align*}
  This validates our result.
  
  \end{proof}
  \begin{cor}
   For $q= \infty,  ~L^{1,q}(\R^n) \subset \mathcal{KS}^{1,p}(\R^n)$ as a dense continuous embedding.
  \end{cor}
 
 \begin{thm}\label{hk11}
 Let $\mathfrak{f}$ be a measurable function on $\Omega \subseteq \R^n$ with extension property. Then $\mathfrak{f} \in \mathcal{KS}^{1,p}(\R^n),~1 \leq p \leq \infty $, if and only if there exists a non negative  function $\mathfrak{g} \in K{S^p}(\R^n)$ such that $$|\mathfrak{f}(x)-\mathfrak{f}(y)| \leq |x-y|\bigg(\mathfrak{g}(x)+\mathfrak{g}(y)\bigg)~a.e..$$
 \end{thm}
 \begin{proof}
 Since the maximal operator is bounded in $\mathcal{KS}^p(\R^n),~p>1,$ we have that if $\mathfrak{f} \in \mathcal{KS}^{1,p}(\Omega)$ where $\Omega = \R^n$ or $\Omega $ is a bounded domain with the extension property, then there exists a non-negative function $\mathfrak{g} \in \mathcal{KS}^p(\Omega)$ such that $$|\mathfrak{f}(x)-\mathfrak{f}(y)| \leq |x-y|\bigg(\mathfrak{g}(x)+\mathfrak{g}(y)\bigg)~a.e..$$ The detailed of the proof is similar \cite[Theorem 1]{Priotr}.
 \end{proof}
 \begin{thm}
 Let $\mathfrak{f}$ be a measurable function on $\Omega \subseteq \R^n$ with extension property then $Lip(\R^n) \subseteq \mathcal{KS}^{1,p}(\R^n).$
 \end{thm}
 \begin{proof}
 Let $\mathfrak{f} \in Lip(\mathbb{R}^n),$ with extension property. Then there exists a bounded linear operator $E:Lip(\Omega) \to Lip(\mathbb{R}^n)$ such thet for every $\mathfrak{f} \in Lip(\Omega),~E_{\mathfrak{f}_{|_\Omega}}=\mathfrak{f}$ a.e.. Hence $|\nabla \mathfrak{f}|(x)= E_{|\nabla \mathfrak{f}|_{|_\Omega}}=f$ a.e.. Hence, $$|\mathfrak{f}(x)-\mathfrak{f}(y)| \leq |x-y|\bigg(\mathfrak{g}(x)+\mathfrak{g}(y)\bigg)~a.e..$$ The Theorem \ref{hk11} gives that $\mathfrak{f} \in \mathcal{KS}^{1,p}(\R^n).$ Hence $Lip(\R^n) \subseteq \mathcal{KS}^{1,p}(\R^n).$
 \end{proof}
 We are coming to Lipschitz type Kuelbs-Steadman spaces on separable metric spaces $\mathcal{X}$ as follows:  let $(\mathcal{X}, d, \mu)$ be a separable metric measure space $(\mathcal{X},d)$ with finite diameter $diam \mathcal{X}= \sup\limits_{x, y \in \mathcal{X}}d(x,y) < \infty$ and a finite positive Borel measure $\mu.$ We can find the class $\mathcal{KS}^{1,p}(\mathcal{X})$  from the Theorem \ref{hk11} as follow:
 \begin{eqnarray*}
 \mathcal{KS}^{1,p}(\mathcal{X})= \big\{\mathfrak{f}: \mathcal{X} \to \mathbb{R} ~|~\mathfrak{f}~{\textit~is~measurable~and}~\exists~E \subset \mathcal{X},~\mu(E)=0~a.e.~and~\mathfrak{g} \in \mathcal{KS}^{p}(\mathcal{X})~ such~ that \\|\mathfrak{f}(x)-\mathfrak{f}(y)| \leq d(x,y)(\mathfrak{g}(x)+\mathfrak{g}(y))~\forall~x,y \in X\setminus E\big\}.
 \end{eqnarray*}
  The space $\mathcal{KS}^{1,p}(\mathcal{X})$ is equipped with the semi norm $$\|\mathfrak{f}\|_{\mathcal{KS}^{1,p}(\mathcal{X})}= \inf\|\mathfrak{g}\|_{\mathcal{KS}^{p}(\mathcal{X})},$$
  where the infimum is taken over all such $\mathfrak{g}$ in the definition of $\mathcal{KS}^{1,p}(\mathcal{X}).$
  
 \begin{thm}
 If $1 \leq p \leq \infty$ then to every $\mathfrak{f} \in \mathcal{KS}^{1,p}(\mathcal{X}),$ there exists the unique $\mathfrak{g} \in \mathcal{KS}^{p}(\mathcal{X})$ which minimizes $\mathcal{KS}^{p}$ norm among the function which can be used in the definition $\|\mathfrak{f}\|_{\mathcal{KS}^{1,p}(\mathcal{X})}.$
 \end{thm}
 \begin{proof}
  Mazur's lemma or Banach-Sak's theorem  gives the existence of a minimizer. The uniquess follows from  uniform convexity of $\mathcal{KS}^{p}(\mathcal{X}).$
  
 \end{proof}
 \begin{rem}
We know that $\parallel \cdot \parallel_{\mathcal{KS}^{1,p}}$ is semi norm on $\mathcal{KS}^{1,p}(\mathcal{X}),$ but it certainly induces the norm on the quotient space $\overbrace{\mathcal{KS}^{1,p}} = \mathcal{KS}^{1,p}/ Const(\mathcal{X})$, where $Const(\mathcal{X})$ is set of all constant function in $\mathcal{X}$. The following corollary is about the completeness of the space $\overbrace{\mathcal{KS}^{1,p}}.$
\end{rem}
\begin{cor}
$\overbrace{\mathcal{KS}^{1,p}}$ is a Banach space with respect to the norm induced from $\mathcal{KS}^{1,p}(\mathcal{X}).$
\end{cor}
 \section{  HK-Sobolev spaces}
   
 Suppose $ k \in \N,~~1 \leq p \leq \infty.$ 
The set of all functions $\mathfrak{f}$ on $\mathbb{R}^n$ such that each mixed partial erivative $$\D^{(\alpha)}(\mathfrak{f}) = \frac{\partial^{|\alpha|} \mathfrak{f}}{\partial x_{1}^{\alpha_1}....\partial x_{n}^{\alpha_n}}$$ exists in the weak sense for every multi-index $\alpha$ with $ |\alpha| \leq k$ is known as the HK-Sobolev space, denoted by $W{S^{k,p}}(\R^n).$

      Therefore the HK-Sobolev  space $W{S^{k,p}}\big(\R^n\big)$ is the space  $$W{S^{k,p}}(\R^n) = \{ \mathfrak{f} \in \mathcal{KS}^{p}(\R^n):~~\D^{\alpha} \mathfrak{f} \in \mathcal{KS}^{p}(\R^n),~\forall |\alpha| \leq k \}.$$  $k $ is degignated as the order of  $W{S^{k,p}}(\R^n).$ 
      
      A norm for $W{S^{k,p}}(\R^n) $ is defined as  $$\|\mathfrak{f}\|_{W{S^{k,p}}(\R^n)}  =\left[\begin{array}{c}\left\{\sum\limits_{|\alpha| \leq k}\|\D^{\alpha} \mathfrak{f} \|_{\mathcal{KS}^p}^{p} \right\}^{\frac{1}{p}}, \mbox{~for~} 1\leq p<\infty;\\
  \max\limits_{|\alpha| \leq k}\|\D^\alpha \mathfrak{f}\|_{\mathcal{KS}^\infty}~~, \mbox{~for~} p=\infty \end{array}\right.$$  
  Recalling $L^p(\R^n) \subset \mathcal{KS}^2(\R^n)$ as dense continuous embeddings.\\ Functional analytic properties of $W{S^{k,p}}(\R^n)$ and $W{S_{0}^{k,p}}(\R^n)$ follow by considering their natural imbedding into the product of $N_k$ copies of $\mathcal{KS}^{p}(\R^n)$ where $N_k$ is the number of multi-indices $\alpha$ satisfying $|\alpha| \leq k.$
 \begin{thm}
 \cite{BH} $W{S^{k,p}}(\R^n)$ is separable for $ 1 \leq p < \infty.$
 \end{thm}
 \begin{thm}
 \cite{BH} $W{S^{k,p}}(\R^n)$ is reflexive for $ 1 < p < \infty.$
 \end{thm}

   \begin{thm}\label{april9}
   For $ 1 \leq q < \infty,$ $W^{k,q}(\R^n) \hookrightarrow W{S^{k,q}}(\mathbb{R}^n).$  
   \end{thm}
   \begin{proof}
    Let $ \mathfrak{f} \in W^{k,q}(\mathbb{R}^n).$ Then we have 
    \begin{align*}
    \|\mathfrak{f}\|_{W{S^{k,q}}(\mathbb{R}^n)} &= \left[ \sum_{|\alpha| \leq k}\| \D^\alpha \mathfrak{f} \|_{\mathcal{KS}^q}^{q}\right]^{\frac{1}{q}} \\&= \left[ \sum_{|\alpha| \leq k}  \sum_{r=1}^{\infty} \tau_r \left| \int_{\mathbb{R}^n} \chi_r(t) \D^\alpha \mathfrak{f}(t) d \lambda_\infty(t)\right|^q\right]^{\frac{1}{q}} \\&\leq \left[ \sum_{|\alpha| \leq k }\sum_{r=1}^{\infty} \tau_r  \int_{\mathbb{R}^n} \chi_r(t)| \D^\alpha \mathfrak{f}(t)|^q d \lambda_\infty (t))^{\frac{1}{q}}\right] \\& \leq \sup_{|\alpha| \leq k}\left\{ \int_{\mathbb{R}^n} \chi_r(t) | \D^\alpha \mathfrak{f}(t) |^q d \lambda_\infty(t)\right\}^{\frac{1}{q}}\leq \|\mathfrak{f}\|_{W^{k,q}(\mathbb{R}^n)}.
    \end{align*}
   \end{proof}
    \subsection{ Lipschitz tpye HK-Sobolev space $W{S^{1,p}}(\R^n)$}
    It is known that Sobolev space $W^{1,p}$ consists of Lipschitz functions. The Theorem \ref{april9} confirm us that   Lipschitz functions are also in $W{S^{1,p}}.$ 
 
  We can define the  Lipschitz type HK-Sobolev spaces as follows:  
  \begin{Def}
  \begin{eqnarray}
   W{S^{1,p}}(\R^n) = \big\{\mathfrak{f} \in \mathcal{KS}^{p}(\R^n):~\nabla \mathfrak{f} \in K{S^p}(\R^n).\big\}
  \end{eqnarray}
  \end{Def}
  It is not hard to see that $W{S^{1,p}}(\R^n)$ is a Banach space with the norm $$\|\mathfrak{f}\|_{W{S^{1,p}}}= \|\mathfrak{f}\|_{\mathcal{KS}^{p}}+ \|\nabla \mathfrak{f}|\|_{\mathcal{KS}^{p}}.$$
  \begin{Def}
  Let $1 \leq p < \infty.$ The HK-Sobolev space with zero boundary values $W{S_{0}^{1,p}}(\R^n)$ is the completion of $C_0^\infty(\R^n)$ with respect to HK-Sobolev norm. Thus $\mathfrak{f} \in W{S_{0}^{1,p}}(\R^n)$ if and only if there exist functions $\mathfrak{f}_i \in C_0^\infty(\R^n),~i=1,2,3,...$ such that $\mathfrak{f}_i \to \mathfrak{f}$ in $W{S^{1,p}}(\R^n)$ as $ i \to \infty.$
  \end{Def}
  The space $W{S_{0}^{1,p}}(\R^n)$ is endowed with the norm of $W{S^{1,p}}(\R^n).$
  \begin{thm}
 The spaces $W{S^{1,p}}(\R^n)$ and $W{S_{0}^{1,p}}(\R^n)$ donot coincide for bounded $\R^n.$ 
 \end{thm}
 \begin{proof}
 The functions in  $W{S_{0}^{1,p}}(\R^n)$ can be approximated by $C_0^\infty(\R^n).$ On the other hand functions of $W{S^{1,p}}(\R^n)$ can be approximated by $C^\infty(\R^n).$
 \end{proof}
 That is 
 \begin{eqnarray*}
 W{S^{1,p}}(\R^n)= \overline{C^\infty(\R^n)}~and~W{S_{0}^{1,p}}(\R^n)= \overline{C_0^\infty(\R^n)},
 \end{eqnarray*}
 where the completions are taken with respect to the HK-Sobolev norm.
 \\We can construct the spaces $W{S_{0}^{k,p}}(\R^n)$ arises by taking the closure of $C_0^k(\R^n)$ in $W{S^{k,p}}(\R^n).$
 \begin{thm}
 The spaces $W{S^{k,p}}(\R^n)$ and $W{S_{0}^{k,p}}(\R^n)$ donot coincide for bounded $\R^n.$ 
 \end{thm}
 For the case $p=2$ i.e. $W{S^{k,2}}(\R^n),~W{S_{0}^{k,2}}(\R^n).$ Taking the inner product on 
 \begin{eqnarray*}
 \bigg< \mathfrak{f}, \mathfrak{g} \bigg>_{W{S^{k,2}}(\R^n)} = \sum\limits_{|\alpha| \leq m} \bigg< D^\alpha \mathfrak{f}, D^\alpha \mathfrak{g}\bigg>_{\mathcal{KS}^2(\R^n)}.
 \end{eqnarray*}
 Observe that $\|\mathfrak{f}\|_{W{S^{k,2}}(\R^n)} = \big< \mathfrak{f}, \mathfrak{f} \big>_{{W{S^{k,2}}^{\frac{1}{2}}(\R^n)}}.$
 

 \subsection{HK-Sobolev spaces on metric measure spaces}
Now, we are ready to formulate HK-Sobolev spaces on a metric measure space as follows:
 \begin{Def}\label{HKP}

  Let $ 1 < p \leq \infty.$ The HK-Sobolev space  on $\mathcal{X}$, denoted by $W{S^{1,p}}(\mathcal{X})$ is defined as $$W{S^{1,p}}(\mathcal{X})= \big\{\mathfrak{f} \in \mathcal{KS}^{1,p}(\mathcal{X}):~\mathfrak{f} \in \mathcal{KS}^{p}(\mathcal{X})\big\}.$$
 \end{Def}
 The space $W{S^{1,p}}(\mathcal{X})$ is equipped with the norm $$\|\mathfrak{f}\|_{W{S^{1,p}}(\mathcal{X})}= \|\mathfrak{f}\|_{\mathcal{KS}^{p}(\mathcal{X})} + \|\mathfrak{f}\|_{\mathcal{KS}^{1,p}(\mathcal{X})}.$$

 \begin{lem}
 In terms of set $\mathcal{KS}^{1,p}(\mathcal{X})= WS^{1,p}(\mathcal{X}).$ 
 \end{lem}
 \begin{proof}
 Let $f \in \mathcal{KS}^{1,p}(\mathcal{X}) $ then by definition $\exists~E \subset \mathcal{X}$ with $\mu(E)=0$ a.e. and $ g \in \mathcal{KS}^p(X)$ such that $|\mathfrak{f}(x)-\mathfrak{f}(y)| \leq d(x,y)(\mathfrak{g}(x)+\mathfrak{g}(y))$ for all $x,y \in \mathcal{X}\setminus E.$\\
 For a fix $y \in \mathcal{X} \setminus E $ with $g(y) < \infty$,
 $$ |\mathfrak{f}(x)| \leq |\mathfrak{f}(x)-\mathfrak{f}(y)|+|\mathfrak{f}(y)| \leq (diam\mathcal{X})(\mathfrak{g}(x)+\mathfrak{g}(y))+|\mathfrak{f}(y)| $$
 Since, $(diam\mathcal{X})(\mathfrak{g}(x)+\mathfrak{g}(y))+|\mathfrak{f}(y)| \in \mathcal{KS}^p(\mathcal{X})$ with respect to $x$, we have $\mathfrak{f}\in \mathcal{KS}^p(\mathcal{X}).$
 \end{proof}
 \begin{thm}
 $WS^{1,p}(\mathcal{X})$ is a Banach space for all $1< p \leq \infty.$
 \end{thm} 
\begin{proof}
Let $\{\mathfrak{f}_n\}$ be a Cauchy sequence in $WS^{1,p}(\mathcal{X})$ and $\mathfrak{f}_n \rightarrow \mathfrak{f}$ ~in~$ \mathcal{KS}^p(\mathcal{X}) .$ We claim that $\mathfrak{f} \in WS^{1,p}(\mathcal{X})$ and that the convergence hold in $WS^{1,p}(\mathcal{X}).$ Let $\{\mathfrak{f}_{n_k}\}$ be a sub sequence such that $\| \mathfrak{f}_{n_{k+1}} - \mathfrak{f}_{n_k} \|_{WS^{1,p}(\mathcal{X})} \leq 2^{-k}.$ Thus $\mathfrak{f}_{n_k} \rightarrow \mathfrak{f}$ a.e. and $\exists~ \mathfrak{g}_{k} \in \mathcal{KS}^p(\mathcal{X})$ such that
\begin{eqnarray}\label{1}
|(\mathfrak{f}_{n_{k+1}} - \mathfrak{f}_{n_k})(x) - (\mathfrak{f}_{n_{k+1}} - \mathfrak{f}_{n_k})(y) | \leq d(x,y)\{\mathfrak{g}_k(x)\mathfrak{g}_k(y)\}
\end{eqnarray} 
and $\|\mathfrak{g}_k \|_{\mathcal{KS}^p(\mathcal{X})} < 2^{-k}.$ If we take $h= \sum_{k=1}^\infty \mathfrak{g}_k ,$ then $\| h \|_{\mathcal{KS}^p(\mathcal{X})} < 1$ and from (\ref{1}) we get for $l > k,$
$$|(\mathfrak{f}_{n_l}- \mathfrak{f}_{n_k})(x)- (\mathfrak{f}_{n_l}- \mathfrak{f}_{n_k})(y) | \leq d(x,y) \left(\sum_{r=k}^\infty \mathfrak{g}_r(x) + \sum_{r=k}^\infty \mathfrak{g}_r(y)\right).$$
Taking limit $l\longrightarrow \infty$, we get $\mathfrak{f} \in \mathcal{KS}^{1,p}(\mathcal{X}) =WS^{1,p}(\mathcal{X})  $ and $\mathfrak{f}_{n_k}\rightarrow \mathfrak{f}$ in $\mathcal{KS}^{1,p}(\mathcal{X})$. It follows that $\mathfrak{f}_n \rightarrow \mathfrak{f}$ in $WS^{1,p}(\mathcal{X})$.
\end{proof} 
We will discuss Poincar\'e type inequality in associate with $\mathcal{KS}^p(\mathcal{X})$ and $\mathcal{KS}^{1,p}(\mathcal{X})$ as follows:
\begin{lem}\label{2}
 \textbf{(Poincar\'e type inequality)} Let $\mathfrak{f} \in WS^{1,p}(\mathcal{X})$ then $\| \mathfrak{f}-\mathfrak{f}_\mathcal{X}\|_{\mathcal{KS}^p(\mathcal{X})} \leq 2(diam ~\mathcal{X})\| \mathfrak{f}\|_{\mathcal{KS}^{1,p}(\mathcal{X})}.$
\end{lem}
\begin{rem}\label{r1}
$\mathfrak{f}_\mathcal{X}$ is defined as the average of $\mathfrak{f}$ on $\mathcal{X}$, $\mathfrak{f}_\mathcal{X}= \mu(x)^{-1}\int_X \mathfrak{f} d\mu$ and is denoted by $\int^*_\mathcal{X} \mathfrak{f} d\mu.$
\end{rem}
 \begin{proof}[\textbf{Proof of the lemma \ref{2}}]
 Let $\mathfrak{g}$ be such a function from the definition of $\|\mathfrak{f}\|_{\mathcal{KS}^{1,p}}$ that $\| \mathfrak{g} \|_{\mathcal{KS}^p(\mathcal{X})} \leq (1+\epsilon) \| \mathfrak{f} \|_{\mathcal{KS}^{1,p}(\mathcal{X})}$. We have 
 $$|\mathfrak{f}(x)- \mathfrak{f}(y)|\leq d(x,y)~(\mathfrak{g}(x)+\mathfrak{g}(y)) \leq (diam ~\mathcal{X})(\mathfrak{g}(x)+\mathfrak{g}(y)).$$
 Therefore,
 \begin{eqnarray}
 |\mathfrak{f}-\mathfrak{f}_X| \leq \int^*_\mathcal{X}|\mathfrak{f}(x)-\mathfrak{f}(y)|d\mu(y) \leq (diam~\mathcal{X})\left(\mathfrak{g}(x)+\int^*_\mathcal{X} \mathfrak{g}~d\mu\right). 
 \end{eqnarray}
 Now, for $1<p<\infty$
 \begin{align*}
 \| \mathfrak{f}-\mathfrak{f}_\mathcal{X} \|_{\mathcal{KS}^p(\mathcal{X})} &= \left( \sum_{k=1}^\infty t_k \left|\int_\mathcal{X} \varepsilon_k(x)(\mathfrak{f}-\mathfrak{f}_\mathcal{X})d\mu \right|^p \right)^{\frac{1}{p}} \leq \left( \sum_{k=1}^\infty t_k \int_\mathcal{X} \varepsilon_k(x)\left|\mathfrak{f}-\mathfrak{f}_\mathcal{X}\right|^pd\mu  \right)^{\frac{1}{p}}\\
 & \leq \left( \sum_{k=1}^\infty t_k \int_\mathcal{X} \varepsilon_k(x)\left\{(diam~\mathcal{X})\left(\mathfrak{g}(x)+\int^*_\mathcal{X} \mathfrak{g}~d\mu\right)\right\}^pd\mu\right)^{\frac{1}{p}}\\
 & \leq (diam\mathcal{X})\left\{\left(\sum_{k=1}^\infty t_k\left|\int_\mathcal{X}\varepsilon_k(x)\mathfrak{g}(x)d\mu\right|^p\right)^{\frac{1}{p}} + \mu(\mathcal{X})^{-1}\left(\sum_{k=1}^\infty t_k \left| 
 \int_\mathcal{X} \varepsilon_k(x)\mathfrak{g}(x)d\mu \right|^p \right)^{\frac{1}{p}}\right\}\\
 & = (diam\mathcal{X})\left( \| \mathfrak{g} \|_{\mathcal{KS}^p(\mathcal{X})} + \frac{\| \mathfrak{g} \|_{\mathcal{KS}^p}(\mathcal{X})}{\mu(\mathcal{X})}\right) \\&\leq 2 (diam \mathcal{X}) \| g \|_{\mathcal{KS}^p(\mathcal{X})} ~\\&\leq 2(1+\epsilon)(diam \mathcal{X})\| \mathfrak{f} \|_{\mathcal{KS}^{1,p}(\mathcal{X})}.
 \end{align*}
 \end{proof}
 \begin{thm}
 The norm $\| \cdot \|_{WS^{1,p}(\mathcal{X})}$ is equivalent with $\| \mathfrak{f}\|^{\divideontimes} = \| \mathfrak{f} \|^{\bullet} +\| \mathfrak{f}\|_{\mathcal{KS}^{1,p}(\mathcal{X})}$, where $\| \cdot \|^{\bullet}$ is any continuous norm on $WS^{1,p}(\mathcal{X})$.
 \end{thm}
 \begin{proof}
 To prove this theorem we need only to show that $WS^{1,p}$ is complete with respect to the norm $\| \cdot\|^{\divideontimes}$.
The rest will follow from Banach's theorem. 

Let $\{\mathfrak{f}_n\}$ be a Cauchy sequence in $WS^{1,p}$ with respect to the norm  $\parallel \cdot\parallel^{\divideontimes}$. Then it is a Cauchy sequence in $\mathcal{KS}^{1,p}$ and by lemma \ref{2} we have $\{\mathfrak{f}_n - (\mathfrak{f}_n)_\mathcal{X}\}$ is a Cauchy sequence in $WS^{1,p}$. Suppose $\mathfrak{f}$ is the limit of $\{\mathfrak{f}_n - (\mathfrak{f}_n)_\mathcal{X}\}$ in $WS^{1,p}$. Since $\parallel \cdot\parallel^{\bullet}$ is continuous on $WS^{1,p}$, we get $\{\mathfrak{f}_n - (\mathfrak{f}_n)_\mathcal{X}\}$ as well as $\{\mathfrak{f}_n\}$ are Cauchy sequence in $\| \cdot\|^{\bullet}$. Hence the sequence $\{(\mathfrak{f}_n)_\mathcal{X}\}$  is convergent to a constant $C$. Now, it goes without saying that $\mathfrak{f}_n \rightarrow \mathfrak{f}+C \in WS^{1,p}$ in $\| \cdot\|^{\divideontimes}$.
\end{proof}

The next theorem talks about the density of Lipschitz functions in $WS^{1,p}(\mathcal{X})$.
\begin{thm}
For every function $ \mathfrak{f} \in WS^{1,p}(\mathcal{X})~~( 1<p < \infty) $ and $\epsilon>0~ \exists$   a Lipschitz function $\mathfrak{g}$ such that, if $E = \{ x \in \mathcal{X}~ |~ \mathfrak{f}(x) \neq \mathfrak{g}(x)\}$, then $\mu(E) < \epsilon$. Moreover, $\| \mathfrak{f} - \mathfrak{g} \|_{WS^{1,p}} < \epsilon$.
\end{thm}
\begin{proof}
The proof is analogous to the proof  of  \cite[Theorem 5]{Priotr}.
\end{proof}

\section{Boundedness of maximal operators on $\mathcal{KS}^p(X)$ and $WS^{1,p}(X)$}
In this section, we will prove that the Maximal operators are bounded in Kuelb-Steadman spaces and HK-Sobolev spaces on metric spaces. We start this section with the following definition.
\end{thm}
\begin{prop}
For a integrable function $\psi: [0, \infty) \to \mathbb{R}$ and a measurable function $\mathfrak{f} : \mathcal{X} \to [0, \infty)$, if $$\Psi(t)= \int_0^t\psi(s)ds.$$
Then \begin{align}\label{p1}
\int_\mathcal{X}\Psi(\mathfrak{f})d\mu = \int_0^\infty \Psi(s)\mu(\{\mathfrak{f}>s\})ds.
\end{align}
\end{prop}
\begin{proof}
The proof the proposition directly follows from Fubini's theorem. For that we assume $A = \{ (x,s) : \mathfrak{f}(x)>s\} \subset \mathcal{X}\times \mathbb{R}$. Then 
$$\mu(\{\mathfrak{f} >s\}) = \int_\mathcal{X}\zeta_Ad\mu$$
where $\zeta_A$ is the characteristic functionn on $A$ and 
$$\int_0^\infty \Psi(s)\zeta_A(s)ds = \int_0^{\mathfrak{f}(x)}\Psi(s)ds = \Psi(\mathfrak{f}(x))$$ Now the results follows from Fubini's theorem.
\end{proof}
\begin{thm}\label{maximal oprerator}
Let $(\mathcal{X}, d, \mu)$ be a separable metric space with doubling measure $\mu,$ then the maximal function maps $\mathcal{KS}^1(\mathcal{X})$ to $\mathcal{KS}^1_w(\mathcal{X})$ and $\mathcal{KS}^p(\mathcal{X})$ to $\mathcal{KS}^p(\mathcal{X})$ for $p > 1$ in the following precise sense: \\
for all $t >0$ if $\mathfrak{f} \in \mathcal{KS}^1(\mathcal{X}),$ we have
\begin{align}\label{M1}
\mu(\{\mathcal{M}\mathfrak{f}> t\}) \leq \frac{C_1}{t}\|\mathfrak{f}\|_{\mathcal{KS}^1(\mathcal{X})}
\end{align}
and for $p>1$, if $\mathfrak{f} \in \mathcal{KS}^p(\mathcal{X})$
\begin{align} \label{M2}
\|\mathcal{M}\mathfrak{f}\|_{\mathcal{KS}^p(\mathcal{X})} \leq C_p\|\mathfrak{f}\|_{\mathcal{KS}^p(\mathcal{X})}
\end{align} 
where the the constants $C_1$ and $C_p$  are only affected by $p$ and the doubling constant $\mu.$
\end{thm}
\begin{proof}
   The basic  covering Theorem \ref{BCL} and the fact that $\mu$ is doubling leads to the inequality (\ref{M1}). We consider restricted maximal function $\mathcal{M}_R\mathfrak{f}(x) = \sup\limits_{0<r<R} \frac{1}{\mu(B)}\int_B|\mathfrak{f}|d\mu$ and then letting $R \to \infty$. With this, we proceed with the notation $\mathcal{M}_R\mathfrak{f} = \mathcal{M}\mathfrak{f}.$
    
    For every $x$ in the set $\{x : \mathcal{M}\mathfrak{f}(x) >t\}$, we can choose a ball $B(x,r)$ such that $$\left| \int_B \chi(x)\mathfrak{f}(x)d\mu \right| > t \mu(B(x,r))$$
   where $\chi(x)$ is the characteristic function on $B.$ Now from this collection taking a countable sub collection $\mathcal{G}$ as in the basic covering theorem and using the doubling property of $\mu$ and disjointness of family $\mathcal{G},$ we get
   \begin{align*}
   \mu(\{\mathcal{M}\mathfrak{f} > t\}) &\leq \sum_\mathcal{G}\mu(5B)\leq C_1\sum_\mathcal{G}\mu(B)\\
   & \leq \frac{C_1}{t}\sum_{\mathcal{G}}\left| \int_B \zeta(x)\mathfrak{f}(x)d\mu \right| \leq \frac{C_1}{t}\|f\|_{\mathcal{KS}^1(\mathcal{X})}.
   \end{align*}
   
   To prove inequality (\ref{M2}), for $\mathfrak{f} \in \mathcal{KS}^p(\mathcal{X})$ we consider
   $$\mathfrak{f} = \mathfrak{f} \cdot \chi_{\{\mathfrak{f}< 1/2\}} + \mathfrak{f}\cdot \chi_{\{\mathfrak{f}>1/2\}} =: g+h$$
   with fixed $t>0$, where $\chi_\star$ denotes the characteristic function. Then $\mathcal{M}\mathfrak{f} \leq \mathcal{M}g +\mathcal{M}h \leq t/2 +\mathcal{M}h$ so that $\{\mathcal{M}\mathfrak{f} > t\} \subset \{\mathcal{M}h < t/2\}$. Now from weak estimate (\ref{M1}) we have, by using (\ref{p1}) with $\Psi(x)= \max(x-\frac{t}{2}, 0)$ and $\psi= \chi_{(t/2,\infty)}$(Characteristic function on $(t/2, \infty)$) that
   \begin{align*}
   \mu(\{\mathcal{M}h > t/2\})&\leq \frac{C}{t} \int_X|h|d\mu \leq \frac{C}{t}\int_{|\mathfrak{f}|>t/2} |\mathfrak{f}| d\mu\\
   &= \frac{C}{t}\left(\int_{|\mathfrak{f}|>t/2}(|\mathfrak{f}|-t/2)d\mu + \frac{t}{2}\mu(\{|\mathfrak{f}|>t/2)\right)\\
   & =\frac{C}{t} \int_{t/2}^\infty \mu(\{|\mathfrak{f}|>s\})ds + C\mu(\{|\mathfrak{f}| >t/2\}) .
   \end{align*}
   
   Now from the formula for $\mathcal{KS}^p$-norm of $\mathcal{M}\mathfrak{f}$ and applying (\ref{p1}) to $\mathcal{M}\mathfrak{f}$ with $\Psi(x)= x^p$ and the the above estimate we get
   \begin{align*}
   \|\mathcal{M}\mathfrak{f}\|_{\mathcal{KS}^p(\mathcal{X})}^p &= \sum\limits_{r=1}^\infty \tau_r \left|\int_\mathcal{X} \chi_r(x)\mathcal{M}\mathfrak{f}(x) d\mu \right|^p \\
   & = \sum\limits_{r=1}^\infty \tau_r \left|\int_0^\infty \chi_r(x)t^{p-1} \mu(\{\mathcal{M}\mathfrak{f} >t \})\right|^p\\
   &\leq \sum\limits_{r=1}^\infty \tau_r \left|\int_0^\infty \chi_r(x)t^{p-1} \mu(\{\mathcal{M}h >t \})\right|^p\\
   &\leq C \sum\limits_{r=1}^\infty \tau_r \left|\int_0^\infty \chi_r(x) t^{p-2} \left(\int_0^\infty \mu(\{\mathfrak{f}>s\})ds + \frac{t}{2}\mu(\{\mathfrak{f}>t/2\})\right)dt\right|^p\\
   &\leq C \sum_{r=1}^\infty \tau_r \left|\int_\mathcal{X}\chi_r(x)\mathfrak{f}(x)d\mu \right|^p = C \|\mathfrak{f}\|_{\mathcal{KS}^p(\mathcal{X})}^p.
   \end{align*}
This implies that $\|\mathcal{M}\mathfrak{f}\|_{\mathcal{KS}^p(\mathcal{X})} \leq C_p \|\mathfrak{f}\|_{\mathcal{KS}^p}(\mathcal{X})$, where $C_p$ depends only on $p$ and the doubling constant of $\mu$.
\end{proof}
\begin{thm}
 The maximal operator $\mathcal{M}$ is bounded in the space $WS^{1,p}(\mathcal{X})$ for $p>1$.
\end{thm}
\begin{proof}
For any $\mathfrak{f} \in WS^{1,p}(\mathcal{X})$, by definition \ref{HKP} and using the boundedness of $\mathcal{M}$ in $\mathcal{KS}^p(\mathcal{X})$ (Theorem \ref{maximal oprerator}) we have that
\begin{align*}
\|\mathcal{M}\mathfrak{f}\|_{WS^{1,p}(\mathcal{X})} &= \|\mathcal{M}\mathfrak{f}\|_{\mathcal{KS}^p(\mathcal{X})}+ \|\mathcal{M}g\|_{\mathcal{KS}^{1,p}(\mathcal{X})}\\
& \leq C_p \|\mathfrak{f}\|_{\mathcal{KS}^p(\mathcal{X})}+ \inf_{g \in \mathcal{KS}^p(\mathcal{X})}\|g\|_{\mathcal{KS}^p(\mathcal{X})}\\& = C_p \|\mathfrak{f}\|_{WS^{1,p}(\mathcal{X})}.
\end{align*}
This completes the proof.
\end{proof}
\section{Conclusions} In this article, we have defined Kuelbs-Steadman spaces on metric measure spaces. Fundamental properties of Kuelbs-Steadman spaces on a metric measure space are also discussed. In the sequel, we extend Kuelbs-Steadman spaces in the approach of Lipschitz functions. Several inclusion properties were discussed. With the above mentioned investigation, we have introduced our main space HK-Sobolev space on metric measure space $\mathcal{X}.$ In continuum Poincar\'e type inequality has been developed in the same section. Finally, as in application we have been investigated the boundedness of Maximal operators are bounded in Kuelb-Steadman spaces and HK-Sobolev spaces on metric spaces.




\begin{thebibliography}{00}
\bibitem{Brezis} H. {\it Brezis}: {\it Analyse Fonctionnelle, Th\'eorie et Applications}, Masson, Paris (1996).
\bibitem{ABH} A {\it Boccuto}, B. {\it  Hazarika}, H. {\it Kalita}: Kuelbs–Steadman Spaces for Banach Space-Valued Measures, Mathematics {\it 8}, 1005, 1-12 (2020).
\bibitem{Amb} L. {\it Ambrosio}, A. {\it Pinamonti}, G. {\it Speight}: Weighted Sobolev spaces on metric measure spaces, J. Reine Angew. Math. {\it 746}, 39-65,  (2019).
	\bibitem{Ivan} I. {\it  Caamano}, J. A. {\it Jaramillo},  A. {\it Prieto}: Characterizing Sobolev spaces of vector-valued functions, J. Math. Anal. Appl. {\it 514}, 1-18, (2022). 
	\bibitem{BH} B. {\it  Hazarika},  H. {\it  Kalita}: HK-Sobolev spaces $W {S^{ k,p}}$ and Bessel Potential, Preprint, 2020
	\bibitem{GILLSURVEY} T. L.{\it Gill}, T. {\it Myers}:  Constructive Analysis on Banach Spaces, Real Anal. Exchange, {\it 44}, 1-36, (2019).
	\bibitem{Gill} T. L. {\it Gill},  W. W. {\it Zachary}: Functional analysis and Feynman operator calculus, Springer, (2016).
	\bibitem{Gilbert} D. {\it Gilbarg},  N. {\it  Trudinger}: \textit{ Elliptic Partial Differential Equations of second Order}, Springer-Verlag. (1983).
\bibitem{GORDON} R. {\it  Gordon}: {\it The integrals of Lebesgue,
Denjoy, Perron, and Henstock}, \rm Amer. Math. Soc., 
Providence, Rhode Island, (1994).
\bibitem{HLP} G. H. {\it  Hardy}, J. E. {\it  Littlewood}, G. {\it  Polya}:  {\it Inequalities}, Cambridge Univ. Press, (1934).
\bibitem{Priotr} P. {\it Hajl\'asz}: Sobolev Spaces on an Arbitrary Metric Space, Potential Analysis, {\it 5},  403-415, (1996).
	\bibitem{Mf} J. {\it Heinonen}: {\it Lectures on analysis on metric spaces}, Springer-Verlag, New York, (2001).
	\bibitem{HK} H. {\it Kalita}: Kuelbs-Steadman spaces with bounded variable exponents, Filomat, {\it 36}, 30-46, (2022).
	 \bibitem{KHM} H. {\it  Kalita}, B. {\it  Hazarika}, T. {\it  Myers}: Kuelbs-Steadman spaces on separable Banach spaces, Facta. Univ. Ser. Math. Inform, {\it  36}, 1064-1077, (2021).
	 \bibitem{Kinnunem} J. {\it Kinnunem}: The Hardy-Littlewood Maximal Function of a Sobolev Function, Israel J. Math., {\it 100}, 117-124, (1997).
	  \bibitem{KUELBS} J. {\it  Kuelbs}, Gaussian measures on a Banach 
space, J. Funct. Anal., {\it 5}, 354-367, (1970).
	\bibitem{WM} W. {\it  Mclean}: {\it Strongly Elliptic Systems and boundary Integral equation}, Cambridge University Press, (2000).
	\bibitem{Stud} P. {\it  Niemiec}: Canonical Banach function spaces generated by
Urysohn universal spaces measures as Lipschitz maps, Stud. Math, {\it 192}, 97-114, (2009).
   
   \bibitem{Stein} E. M. {\it Stein}:  {\it Singular Integrals and Differentiability Properties of Functions}, Princeton University Press, (1970).
 \bibitem{Sobolev} S. L. {\it Sobolev}: Sur un th\'or\'me d\'analyse fonctionnelle Recueil Mathmatique (Matematicheskii Sbornik) (in Russian and French), {\it 4}(46), {\it 3}, 47-497, (1938).
  
\bibitem{TW}  A. {\it Tveito}, R. {\it Winther}: {\it Introduction to Partial Differential Equations: A Computational Approach}. Text in Applied Mathematics, Springer, Berlin (2009).

\end{thebibliography}
 \end{document}